\documentclass[a4paper,12pt]{article}
\usepackage{amssymb,amsmath,amsthm,latexsym}
\usepackage{amsfonts}
	\usepackage{amsfonts}
\usepackage{graphicx}
\usepackage[pdftex,bookmarks,colorlinks=false]{hyperref}
\usepackage{verbatim}
\usepackage{caption}
\usepackage{subcaption}

\newtheorem{theorem}{Theorem}[section]

\newtheorem{corollary}[theorem] {Corollary}
\newtheorem{definition}[theorem]{Definition}

\newtheorem{lemma} [theorem]{Lemma}

\newtheorem{proposition}[theorem]{Proposition}
\newtheorem{remark}[theorem]{Remark}

\title{\bf A Characterisation of Strong Integer Additive Set-Indexers of Graphs}
\author{{\bf N K Sudev \footnote{Department of Mathematics, Vidya Academy of Science \& Technology, Thalakkottukara, Thrissur - 680501, email: {\em sudevnk@gmail.com}}} and {\bf K A Germina\footnote{Department of Mathematics, School of Mathematical \& Physical Sciences, Central University of Kerala, Kasaragod, email:{\em srgerminaka@gmail.com}}}}
\date{}
\begin{document}
\maketitle

\begin{abstract}
An integer additive set-indexer is defined as an injective function $f:V(G)\rightarrow 2^{\mathbb{N}_0}$ such that the induced function $g_f:E(G) \rightarrow 2^{\mathbb{N}_0}$ defined by $g_f (uv) = f(u)+ f(v)$ is also injective, where $f(u)+f(v)$ is the sumset of $f(u)$ and $f(v)$. If $g_f(uv)=k~\forall~uv\in E(G)$, then $f$ is said to be a $k$-uniform integer additive set-indexers. An integer additive set-indexer $f$ is said to be a strong integer additive set-indexer if $|g_f(uv)|=|f(u)|.|f(v)|~\forall ~ uv\in E(G)$. We already have some  characteristics of the graphs which admit strong integer additive set-indexers. In this paper, we study the characteristics of certain graph classes, graph operations and graph products that admit strong integer additive set-indexers.
\end{abstract}
\textbf{AMS Subject Classification No.: 05C78}

\section{Introduction}

For all  terms and definitions, not defined specifically in this paper, we refer to \cite{FH}. Unless mentioned otherwise, all graphs considered here are simple, finite and have no isolated vertices.

Let $\mathbb{N}_0$ denote the set of all non-negative integers. For all $A, B \subseteq \mathbb{N}_0$, the sum of these sets is denoted by  $A+B$ and is defined by $A + B = \{a+b: a \in A, b \in B\}$. The set $A+B$ is called the {\em sumset} of the sets $A$ and $B$. Also, we have $2A=A+A$.

If either $A$ or $B$ is countably infinite, then their sumset is also countably infinite. Hence, the sets we consider here are all finite sets of non-negative integers. The cardinality of a set $A$ is denoted by $|A|$. 

\begin{definition}\label{D2}{\rm
\cite{GA} An {\em integer additive set-indexer} (IASI, in short) is defined as an injective function $f:V(G)\rightarrow 2^{\mathbb{N}_0}$ such that the induced function $g_f:E(G) \rightarrow 2^{\mathbb{N}_0}$ defined by $g_f (uv) = f(u)+ f(v)$ is also injective}.
\end{definition}

\begin{definition}{\rm
\cite{GS2} If a graph $G$ has a set-indexer $f$ such that $|g_f(uv)|=|f(u)+f(v)|=|f(u)|.|f(v)|$ for all vertices $u$ and $v$ of $G$, then $f$ is said to be a {\em strong IASI} of $G$.} 
\end{definition}

\begin{definition}\label{D6}{\rm
\cite{GS2} If $G$ is a graph which admits a $k$-uniform IASI and $V(G)$ is $l$-uniformly set-indexed, then $G$ is said to have a {\em $(k,l)$-completely uniform IASI} or simply a {\em completely uniform IASI}.}
\end{definition}

We use the notation $A<B$ in the sense that $A\cap B=\emptyset$. We notice that the relation $<$ is symmetric, but need not be reflexive and transitive. By the sequence $A_1<A_2<A_3<\ldots <A_n$, we mean that the given sets are pairwise disjoint.

\begin{lemma}\label{L-RDS}
\cite{GS2} Let $A$, $B$ be two non-empty subsets of $\mathbb{N}_0$. Then, $|A+B|=|A|.|B|$ if and only if their difference sets, denoted by $D_A$ and $D_B$ respectively, follow the relation $D_A<D_B$.
\end{lemma}

\begin{theorem}\label{TSK}
\cite{GS2} Let each vertex $v_i$ of the  complete graph $K_n$ be labeled by the set $A_i\in 2^{\mathbb{N}_0}$. Then $K_n$ admits a strong IASI if and only if there exists a finite sequence of sets $D_1<D_2<D_3<\cdots,<D_n$ where each $D_i$ is the set of all differences between any two elements of the set $A_i$.
\end{theorem}

\begin{theorem}\label{TSS}
\cite{GS2} If a graph $G$ admits a strong IASI then its subgraphs also admit strong IASI. 
\end{theorem}

\begin{corollary}\label{TSC}
\cite{GS2} A connected graph $G$ (on $n$ vertices) admits a strong IASI if and only if each vertex $v_i$ of $G$ is labeled by a set $A_i$ in $2^{\mathbb{N}_0}$ and there exists a finite sequence of sets $D_1<D_2<D_3< \cdots <D_m$, where $m\le n$ is a positive integer and each $D_i$ is the set of all differences between any two elements of the set $A_i$. 
\end{corollary}

\section{New Results on Strong IASI Graphs}
\subsection{Strong IASI of Certain Graph Classes}
Based on the results in the above section, we introduce the following notions. 

The set $D_i$ of all differences between two elements of the set $A_i$ is called the {\em difference set} of $A_i$ and the relation $<$ is called the {\em difference relation} on $G$. We call the sequence of difference sets, mentioned in Corollary \ref{TSC} a {\em chain} of  difference sets. The {\em length} of a chain is the number of difference sets in that chain. 

For any strong IASI graph with $m$ vertices and $n$ edges, there are $m$ difference sets, one each for each vertex and $n$ relations $<$, one each corresponding to each edge of $G$. Since $G$ is a strong IASI graph, if two vertices are adjacent in $G$, then there exists the difference relation between the difference sets of the set-labels of those vertices. These relations forms one or more chain of difference sets. The length of a chain of difference sets is noteworthy and hence we make the following notion.

\begin{definition}{\rm
The {\em nourishing number} of a set-labeled graph is the minimum length of the maximal chain of difference sets in $G$. The nourishing number of a graph $G$ is denoted by $\varkappa(G)$.}
\end{definition}

The study about nourishing number of different graphs and graph classes arises much interest. In this section, we discuss about the nourishing number of various graphs.

By Theorem \ref{TSK}, we have the following result.

\begin{theorem}
The nourishing number of a complete graph is $n$. That is, $\varkappa(K_n)=n$.
\end{theorem}

We proceed with the following theorem on the nourishing number of bipartite graphs.

\begin{theorem}
The nourishing number of a bipartite graph is $2$.
\end{theorem}
\begin{proof}
Let $G$ be a bipartite graph with bipartition $(X,Y)$ which admits a strong IASI. Let $X=\{u_1,u_2,u_3,\ldots, u_m\}$ and $Y=\{v_1,v_2,v_3,\ldots, v_n\}$. Define $f:V(G)\to 2^{\mathbb{N}_0}$ such that for two adjacent vertices $u_i\in X$ and $v_j\in Y$, the difference sets of $f(u_i)$ and $f(v_j)$, denoted by $D_{u_i}$and $D_{v_j}$ respectively, hold the relation $D_{u_i}<D_{v_j}$. Hence, corresponding to each edge in $G$, there exists a difference relation between the set-labels of its end vertices. Hence, $f$ is a strong IASI and the minimum length of the maximal chain in $G$ is $2$. That is, $\varkappa(G)=2$.
\end{proof}

\begin{proposition}\label{P-NN1}
The nourishing number of a triangle-free graph is $2$.
\end{proposition}
\begin{proof}
Since $G$ is a triangle-free graph, $G$ can not be a complete graph. Since the relation $<$ is not transitive, we can not find a chain consisting of three or more difference sets corresponding to the set-labels of vertices of $G$. Hence $\varkappa(G)=2$.
\end{proof}

A clique of graph $G$ is a complete subgraph of $G$. We recall that the clique number $\omega(G)$ of a graph is the number of vertices in a maximal clique in $G$.

\begin{proposition}\label{P-NN2}
The nourishing number of a graph $G$ is the clique number $\omega(G)$ of $G$.
\end{proposition}
\begin{proof}
Let $H$ be the maximal clique of a given graph $G$. Then, $H$ is complete graph. Hence, by Theorem \ref{TSK}, there exists a chain of difference sets $D_1<D_2<D_3<\ldots <D_r$, where $r=|V(H)|$. Moreover, no cliques in $G-H$ can have more vertices in $G$ than $H$. Therefore, $\varkappa(G)=r=\omega(G)$.
\end{proof}

From \ref{P-NN1} and \ref{P-NN2}, we propose the following result.

\begin{theorem}
For any strong IASI graph $G$, the nourishing number $\varkappa(G)$ is given by $\varkappa(G)=|V(H)|=\omega(G)$, where $H$ is the maximal clique of $G$.
\end{theorem}

Interesting questions that arise in this context are about the nourishing number of different graph operations. In the following discussions, we address these problems. First, we check the admissibility of strong IASIs by the union of two graphs and its nourishing number in the following results.

\subsection{Strong IASIs of Graph Operations}

By the term {\em last vertex} of a subgraph $H$ of $G$, we mean a vertex $v\in V(H)$ whose adjacent vertices in $G$ are not in $V(H)$.

\begin{theorem}\label{T-SIGU1}
The union $G_1\cup G_2$ of two graphs $G_1$ and $G_2$, admits a strong IASI if and only if both $G_1$ and $G_2$ admit strong IASIs.
\end{theorem}
\begin{proof}
Let $G_1$ and $G_2$ admit strong IASIs say $f_1$ and $f_2$. If $G_1$ and $G_2$ are disjoint graphs, then we observe that the IASI $f$ defined by $f(v)=f_i(v)$ if $v\in V(G_i),i=1,2$ is a strong IASI for $G_1\cup G_2$. 

If $G_1$ and $G_2$ are not disjoint graphs, then re-label the vertices in $G_1\cap G_2$ in such a way that the difference sets of the set-labels of adjacent vertices in $G_1\cap G_2$  and the last vertices of $G_1\cap G_2$ and their adjacent vertices in $G_1$ and $G_2$ follow the relation `$<$'. Hence, $f$ is a strong IASI of $G_1\cup G_2$.
Conversely, assume that $G_1\cup G_2$ admits a strong IASI. Hence, by Theorem \ref{TSS}, being the subgraphs of a strong IASI graph $G_1\cup G_2$, $G_1$ and $G_2$ admit the (induced) strong IASIs, $f|_{G_1}$ and $f|_{G_2}$, where $f|_{G_i}$ is the restriction of $f$ to the (sub)graph $G_i$. 
\end{proof}

\begin{theorem}\label{T-NNGU1}
Let $G_1$ and $G_2$ be two strong IASI graphs. Then, $\varkappa(G_1\cup G_2)\ge max\{\varkappa (G_1),\varkappa(G_2)\}$.
\end{theorem}
\begin{proof}
Let $G_1$ and $G_2$ be two strong IASI graphs. Let $H_1$ and $H_2$ be the maximal cliques in $G_1$ and $G_2$ respectively. If $G_1$ and $G_2$ are disjoint, so are $H_1$ and $H_2$. Without loss of generality, let $|V(H_1)\geq |V(H_2)|$. Then, $H_1$ is the maximal clique in $G_1\cup G_2$. Hence, $\varkappa(G_1\cup G_2)=\varkappa(G_1)$. Therefore, in general, for disjoint graphs $G_1$ and $G_2$, $\varkappa(G_1\cup G_2)= max\{\varkappa (G_1),\varkappa(G_2)\}$. If $G_1$ and $G_2$ are not disjoint, then there may exist a subgraph $H'_1$, not necessarily complete, in $G_1$ and a subgraph $H'_2$, not necessarily complete, in $G_2$ such that $H'_1\cup H'_2=K_l$, where $l\ge |V(H_1)|, |V(H_2)|$. In this case, $\varkappa(G_1\cup G_2)\ge max\{\varkappa (G_1),\varkappa(G_2)\}$. This completes the proof.
\end{proof}

From Theorem \ref{T-NNGU1}, we observe the following theorem.

\begin{theorem}
Let $G_1$ and $G_2$ be two strong IASI graphs. Then, $\varkappa(G_1\cup G_2)= max\{\varkappa (G_1),\varkappa(G_2)\}$ if $G_1\cap G_2$ is triangle-free.
\end{theorem}
\begin{proof}
Every complete graph $K_n$ with more than two vertices contains triangles. Hence, since, $G_1\cap G_2$ is triangle-free, it does not contain any clique. Therefore, there does not exist a subgraph $H'_i$ in $G_i, i=1,2$ such that $H'_1\cup H'_2=K_l, l>2$. Hence, $\varkappa(G_1\cup G_2)= max\{\varkappa (G_1),\varkappa(G_2)\}$.
\end{proof}

The following theorems check the admissibility of strong IASI by the join of two strong IASI graphs and its nourishing number.

\begin{theorem}
Let $G_1(V_1,E_1)$ and $G_2(V_2,E_2)$ be two strong IASI graphs. Then, their join $G_1+ G_2$ admits a strong IASI if and only if the difference set of the set-label of every vertex in $G_1$ is disjoint from the difference sets of the set- labels of all vertices of $G_2$. 
\end{theorem}
\begin{proof}
Let $V(G_1)=\{u_1,u_2,u_3,\ldots,u_m\}$ and $V(G_2)=\{v_1,v_2,v_3,\ldots,v_n\}$. Now let $E_3=\{u_iv_j:u_i\in V(G_1),v_j\in V(G_2)\}$. Let $G_3$ be the subgraph of $G_1+G_2$ with the edge set $E_3$. Therefore, $G_1+G_2=G_1\cup G_2\cup G_3$. Also, let $f_1, f_2, f_3$ be the IASIs defined on $G_1,G_2,G_3$ respectively. Given that $f_1$ and $f_2$ are strong IASIs. Let $f$ be an IASI defined on $G_1+G_2$ by $f(v) = f_i(v)~ \text{if} ~v\in V(G_i), i=1,2,3$. Then, by Theorem \ref{T-SIGU1}, $f$ is a strong IASI if and only if $f_3$ is a strong IASI. 

First assume that $f$ is a strong IASI. Then, $f_3$ is also a strong IASI. Hence, $|g_{f_3}(u_iv_j)|=|f_3(u_i)|.|f_3(v_j)|~ \forall~ u_i\in V(G_1), v_j\in V(G_2)$. By Lemma \ref{L-RDS}, the difference sets  of the set-labels of these vertices follow the relation $D_{u_i}<D_{v_j}, ~ \forall~ u_i\in V(G_1), v_j\in V(G_2)$.

Conversely, assume that $D_{u_i}<D_{v_j}, ~ \forall~ u_i\in V(G_1), v_j\in V(G_2)$. Then, $|g_{f_3}(u_iv_j)|=|f_3(u_i)|.|f_3(v_j)|~ \forall~ u_iv_j\in E_3$. Therefore, $f_3$ is a strong IASI on $G_3$ and hence $f$ is a strong IASI on $G_1+G_2$. This completes the proof.
\end{proof}

\begin{theorem}
Let $G_1$ and $G_2$ be two strong IASI graphs. Then, $\varkappa(G_1+G_2)=\varkappa(G_1)+\varkappa(G_2)$. 
\end{theorem}
\begin{proof}
Let $H_1$ be a maximal clique with order $m$ in $G_1$ and $H_2$ be a maximal clique with order $n$ in $G_2$.  Since every vertex of $G_1$ is joined to every vertex of $G_2$ in $G_1+G_2$, it is so in $H_1+H_2$ also. Since $H_1$ and $H_2$ are cliques, they are complete graphs. Hence, every vertex of $H_1$ is adjacent to all other vertices of $H_1$ and all vertices of $H_2$. Similarly, every vertex of $H_2$ is adjacent to all other vertices of $H_2$ and all vertices of $H_1$. Hence, $H_1+H_2$ is an $(m+n-1)$-regular graph on $m+n$  vertices. Therefore, $H_1+H_2$ is a complete graph and hence is a clique in $G_1+G_2$. Since $H_1$ and $H_2$ are maximal cliques, $H_1+H_2$ is maximal in $G_1+G_2$. Hence, $\varkappa(G_1+G_2)=m+n=\varkappa(G_1)+\varkappa(G_2)$. 
\end{proof}

Next, we discuss about the nourishing number of the complement of a strong IASI graph $G$.  A graph $G$ and its complement $\bar{G}$ have the same set of vertices and hence $G$ and $\bar{G}$ have the same set-labels for their corresponding vertices. We observe that the strong IASIs, except some, defined on $G$ do not induce strong IASI on $\bar{G}$. A set-labeling of $V(G)$ that defines a strong IASI for both the graphs $G$ and its complement $\bar{G}$ may be called a {\em strongly concurrent set-labeling}. The set-labels of the vertex set of $G$ mentioned in this section are strongly concurrent. Hence, we propose the following results.

\begin{proposition}\label{P-SICG1}
Let $G$ be a strong IASI graph on $n$ vertices and let $\bar{G}$ be its complement. Then, $\bar{G}$ admits a strong IASI if and only if the length of the chain of difference sets of set-labels of vertices in $G$ or in $\bar{G}$ is $n$. 
\end{proposition}
\begin{proof}
We have $G\cup\bar{G}=K_n$. Therefore, the length of the chain of difference sets in $G\cup\bar{G}$ is $n$. Since $\varkappa(G\cup \bar{G})\ge max\{\varkappa (G),\varkappa(\bar{G})\}$, the length of the chain of difference sets of set-labels of vertices in $G$ or in $\bar{G}$ must be $n$.

Conversely, assume that the length of the chain of difference sets of set-labels of vertices in $G$ or in $\bar{G}$ is $n$. Then, length of the chain of difference sets in $G\cup \bar{G}$ is $n$, which is the maximum possible length of a chain of difference sets. Therefore, both $G$ and $\bar{G}$ admit strong IASI under the same set-labels for the vertices of $G$. This completes the proof.
\end{proof}

\begin{corollary}
If $G$ is a self-complementary graph on $n$ vertices, which admits a strong IASI, then $\varkappa(G)=n$.
\end{corollary}
\begin{proof}
Since $G$ is self complementary, we have $G\cong \bar{G}$ and hence $\varkappa(G)=\varkappa(\bar{G})$. Therefore, $n=\varkappa(G\cup \bar{G})=max \{\varkappa(G), \varkappa\bar{G}\}=\varkappa(G)$. That is, $\varkappa(G)=\varkappa(\bar{G})=n$. Therefore, $G$ and $\bar{G}$ admit strong IASI. This completes the proof.
\end{proof}

\subsection{Strong IASIs of Graph Products}

In graph products, we may have to make certain number of copies of the graphs and to make suitable of attachments between these copies and the given graphs. Therefore, we have to establish a suitable IASI, if exists, to each of these copies. Hence, we make the following remark.

\begin{remark}\label{R-SICoG}
{\rm Let $G_i$ is a copy of a given graph $G$, which appears in a graph product. Let $n.A=\{na_i: a_i\in A\}$, for $n\in \mathbb{N}_0$. Note that $n.A\ne nA$. If $f$ is a strong IASI on $G$, then the $i$-th copy of $G$ denoted by $G_i$ has the set-label $f_i$ where $f_i(v_i)=r.f(v),~ r\in \mathbb{N}$, where $v_i$ is the vertex in $G_i$ corresponding to the vertex $v$ in $G$. We observe that if two sets $A$ and $B$ are disjoint, then $n.A$ and $n.B$ are also disjoint. Hence, if $f$ is a strong IASI of $G$, then $f_i$ is a strong IASI of $G_i$.}
\end{remark}

In this section, we verify the admissibility of strong IASI by the cartesian product of two graphs. By the term {\em product of graphs} we mean the cartesian product of graphs. we recall the definition the cartesian product of two graphs as follows.

Let $G_1(V_1,E_1)$ and $G_2(V_2,E_2)$ be two graphs.Then, the {\em cartesian product} or simply {\em product} of $G_1$ and $G_2$, denoted by $G_1\Box G_2$, is the graph with vertex set $V_1\times V_2$  defined as follows. Let $u=(u_1, u_2)$ and $v=(v_1,v_2)$ be two points in $V_1\times V_2$. Then, $u$ and $v$ are adjacent in $G_1\Box G_2$ whenever [$u_1=v_1$ and $u_2$ is adjacent to $v_2$] or [$u_2=v_2$ and $u_1$ is adjacent to $v_1$]. If $|V_i|=p_i$ and $E_i=q_i$ for $i=1,2$, then $|V(G_1\Box G_2)|=p_1p_2$  and $|E(G_1\Box G_2)|=p_1q_2+p_2q_1$.

\begin{remark}\label{R-CP2G}{\rm
We observe that the product $G_1\Box G_2$ is obtained as follows. Make $p_2$ copies of $G_1$. Denote these copies by $G_{1_i}, 1\le i\le p_2$, which corresponds to the vertex $v_i$ of $G_2$. Now, join the corresponding vertices of two copies $G_{1_i}$ and $G_{1_j}$ if the corresponding vertices $v_i$ and $v_j$ are adjacent in $G_2$. Thus, we view the product $G_1\Box G_2$ as a union of $p_2$ copies of $G_1$ and a finite number of edges connecting two copies $G_{1_i}$ and $G_{1_j}$ of $G_1$ according to the adjacency of the corresponding vertices $v_i$ and $v_j$ in $G_2$, where $1\le i,j\le p_2, i\neq j$. 

Hence, we make the following inferences on the admissibility of strong IASI by the product of two strong IASI graphs.}
\end{remark}

\begin{proposition}
Let $G_1$ and $G_2$ be two strong IASI graphs. Then, the product $G_1\Box G_2$ admits a strong IASI if and only if the set-labels of corresponding vertices different copies of $G_1$ which are adjacent in $G_1\Box G_2$ are disjoint.   
\end{proposition}
\begin{proof}
Assume that the graphs $G_1$ and $G_2$ admit strong IASIs, say $f$ and $g$ respectively and $G_1\Box G_2$ admits a strong IASI, say $F$. Since $f$ is a strong IASI on $G$, then by Remark \ref{R-SICoG}, the function $f_i$ defined on the $i$-th copy $G_{1_i}$ of $G_1$, by $f_i(v_i)=r.f(v)$, for some positive integer $r$, is a strong IASI on $G_{1_i}$. If the corresponding vertices of two copies $G{1_i}$ and $G_{1_j}$ are adjacent in $G_1\Box G_2$, then $|g_F(v_iv_j)|=|F(v_i)|.|F(v_j)|=|f_i(v_i)|.|f_j(v_j)|, ~\forall~ v_i\in G_{1_i}, v_j\in G_{1_j}$. Hence, $D_{v_i}<D_{v_j}$. That is $f_i(v_i)$ and $f_j(v_j)$ are disjoint.

Conversely, assume that the set-labels of corresponding vertices different copies of $G_1$ which are adjacent in $G_1\Box G_2$ are disjoint. Also, the function $f_i$ defined on the $i$-th copy $G_{1_i}$ of $G_1$, by $f_i(v_i)=r.f(v)$, for some positive integer $r$, is a strong IASI on $G_{1_i}, 1\le i \le |V(G_2)|$. Hence, the difference sets of the set- labels of all the adjacent vertices in $G_1\Box G_2$ are disjoint. Therefore, $G_1\Box G_2$ admits a strong IASI. 
\end{proof}

\begin{proposition}
Let $G_1$ and $G_2$ be two graphs which admit strong IASIs. Then, $\varkappa(G_1\Box G_2) = {\text{max}}\{\varkappa(G_1),\varkappa(G_2)\}$.
\end{proposition}
\begin{proof}
Let $H_1$ and $H_2$ be the maximal clique in $G_1$ and $G_2$ respectively. Without loss of generality, let $H_1$ be greater than $H_2$ in terms of the number of vertices in them. For $1\le i\le |V(G_2)|$, let $H_{1_i}$ be the copy of $H_1$ in $G_{1_i}$ and is the maximal clique in $G_{1_i}$. Now, observe that no vertex of another copy $G_{1_j}$ is adjacent to all the vertices of $H_{1_i}$. Since, all cliques $H_{1_i}$ are isomorphic, for $1\le i\le |V(G_2)|$, $H_{1_i}$ is the maximal clique in $G_1\Box G_2$. Hence, $\varkappa(G_1\Box G_2)=|V(H_{1_i})|=V(H_1)|$. Therefore, in general, $\varkappa(G_1\Box G_2) = {\text{max}}\{\varkappa(G_1),\varkappa(G_2)\}$.
\end{proof}

Next, let us recall the definition of corona of two graphs. The {\em corona} of two graphs $G_1$ and $G_2$, denoted by $G_1\odot G_2$, is the graph obtained taking one copy of $G_1$ (which has $p_1$ vertices) and $p_1$ copies of $G_2$ and then joining the $i$-th point of $G_1$ to every point in the $i$-th copy of $G_2$.

The following results establish the admissibility of strong IASI by the corona of two strong IASI graphs and its nourishing number.

\begin{theorem}
Let $G_1(V_1,E_1)$ and $G_2(V_2,E_2)$ be two strong IASI graphs. Then, their corona $G_1\odot G_2$ admits a strong IASI if and only if the difference set of the set-label of every vertex in $G_1$ is disjoint from the difference sets of the set- labels of all vertices of the corresponding copy of $G_2$. 
\end{theorem}
\begin{proof}
Let $f$ and $g$ be the strong IASI on $G_1$ and $G_2$ respectively. Let $g_i=r.g, 1\le i\le |V(G_1)|$, $r$ being a positive integer,  be an IASI defined on the $i$-th copy $G_i$ of $G$. Then, by Remark \ref{R-SICoG}, $g_i$ is a strong IASI on $G_i$ for all $i=1,2,3,\ldots, |V(G_1)|$. 

First assume that $G_1\odot G_2$ admits a strong IASI. Also, each copy of $G_2$ admits strong IASIs. Since $G_1\odot G_2$ admits a strong IASI, the difference set of the set-label of each vertex $u_i$ of $G_1$ and the difference set of the set-label of each vertex $v_{j_i}$ of $G_{2_i}$, where $1\le i\le |V(G_1), 1\le j\le |V(G_2)$, must be disjoint. 

Conversely, assume that the difference set of the set-label of every vertex in $G_1$ is disjoint from the difference sets of the set- labels of all vertices of the corresponding copy of $G_2$. Since each copy of $G_2$ is also strong IASI graph, for every pair of adjacent vertices $u,v$ in $G_1\odot G_2$, the difference sets $D_u$ and $D_v$ hold the relation $D_u<D_v$. Hence, $G_1\odot G_2$ admits a strong IASI. 
\end{proof}

\begin{proposition}
If $G_1$ and $G_2$ are two strong IASI graphs, then
\[ \varkappa(G_1\odot G_2) = \left\{
  \begin{array}{l l}
    \varkappa(G_1) & \quad \text{if $\varkappa(G_1)> \varkappa(G_2)$}\\
    \varkappa(G_2)+1 & \quad \text{if $\varkappa(G_2)> \varkappa(G_1)$}
  \end{array} \right.\]
\end{proposition}
\begin{proof}
Let $H_1$ and $H_2$ be the maximal cliques in $G_1$ and $G_2$ respectively. Then, $H_{2_i}$ is the copy of $H_2$ in $G_{2_i}$, which is maximal in $G_{2_i}$. Since the vertex $u_i$ of $H_1$ is adjacent to all vertices of the copy $H_{2_i}$ in $G_1\odot G_2$, we can find $|V(G_1)|$ cliques in $G_1\odot G_2$ with clique number $1+|V(H_{2_i}|=1+|V(H_2)|=1+\varkappa(G_2)$.

If $\varkappa(G_1)> \varkappa(G_2)$, then clearly, $\varkappa(G_1\odot G_2)=\varkappa(G_1)$. If  $\varkappa(G_1)< \varkappa(G_2)$, then the maximal clique in $G_1\odot G_2$ is $H_{2_i}+\{u_i\}$. Therefore, $\varkappa(G_1\odot G_2)=1+\varkappa(G_2)$. This completes the proof.
\end{proof}

\section{Conclusion}
In this paper, we have discussed about the admissibility of strong IASIs by certain graph classes, graph operations and graph products. We have established some results on the nourishing number of the graphs and graph operations. The admissibility of strong IASI by various other graph operations and graph products and finding the corresponding nourishing numbers are remained to be estimated.

More properties and characteristics of strong IASIs, both uniform and non-uniform, are yet to be investigated. The problems of establishing the necessary and sufficient conditions for various graphs and graph classes to have certain IASIs still remain unsettled. All these facts highlight a great scope for further studies in this area.

\end{document}